\documentclass[11pt]{article}
\usepackage[a4paper, margin=25mm]{geometry}
\usepackage{amsfonts}
\usepackage{amssymb}
\usepackage{amsmath}
\usepackage{amsthm} 
\usepackage{pifont}
\usepackage[dvipdfmx]{graphicx}
\usepackage{float}
\usepackage{enumerate}
\usepackage{ulem}
\usepackage{color}
\usepackage[dvipsnames]{xcolor}
\usepackage{mathrsfs}
\usepackage{fancybox}
\usepackage{bbm}
\usepackage{mathtools} 
\newcommand\indep{\mathrel{\mathpalette\independenT\perp}}
\newcommand\independenT[2]{%
  \mathrel{\rlap{\scalebox{0.8}[1]{$#1#2$}}\mkern3mu\scalebox{0.8}[1]{$#1#2$}}%
}
\theoremstyle{definition}
\newtheorem{theorem}{Theorem}

\newtheorem{definition}{Definition}
\newtheorem{example}{Example}
\newtheorem{lemma}{Lemma}
\newcommand{\at}{\alpha t}
\newcommand{\refset}{\noindent \hangindent=10pt \hangafter=1}
\newcommand{\dd}{\mathrm{d}}
\newcommand{\ee}{\mathrm{e}}

\newcommand{\ii}{\mathrm{i}}

\newcommand{\E}{{\rm E}}
\newcommand*{\Cdot}{\raisebox{-0.25ex}{\scalebox{1.4}{${\cdot}$}}}
\newcommand{\1}{\mbox{1}\hspace{-0.25em}\mbox{l}} 
\newcommand{\mGamma}[2]{
\raise1.5pt\hbox{$\displaystyle \mathop{\Gamma}^{\rm m}$}
\hspace{-1.5pt}_{#1}^{~#2}}

\newcommand{\Levy}{L\'{e}vy }
\allowdisplaybreaks[4]  
\usepackage[numbers]{natbib}
\title{Bayesian Prediction in Gamma Models: \\
Admissibility and Infinitesimal Prediction}
\author{Fumiyasu Komaki \\
Department of Mathematical Informatics, The University of Tokyo\\
RIKEN Center for Brain Science \\
\texttt{komaki@mist.i.u-tokyo.ac.jp}}
\date{}
\begin{document}

\maketitle

\begin{abstract}
We study estimation and prediction in the Gamma model $\mathrm{Ga}(\alpha,\beta)$, where the shape parameter $\alpha$ is known and the scale parameter $\beta$ is unknown, under the Kullback--Leibler loss. For $\alpha\le1$, all scale-invariant estimators of $\beta$ have infinite risk, indicating a qualitative change in the estimation problem at the boundary $\alpha=1$.

Our main result is that the Bayesian predictive density based on the Jeffreys prior is admissible for all $\alpha>0$. This resolves the admissibility problem for Bayesian predictive densities in Gamma models. As a related result, we also establish the admissibility of the corresponding Bayesian estimator for $\alpha>1$.

To prove the predictive admissibility result, we develop an infinitesimal prediction framework based on Gamma processes. This framework naturally leads to a Kullback--Leibler loss for \Levy densities and establishes a connection between predictive distributions and \Levy measures. Under the resulting loss, the Bayesian predictive \Levy density is shown to be the posterior mean \Levy density.

Unlike the normal and Poisson models, infinitesimal prediction in the Gamma model does not reduce to parameter estimation. Instead, it reduces to the estimation of a \Levy density. We relate this phenomenon to mean mixture curvature and discuss it from an information-geometric viewpoint.
\end{abstract}

\noindent
\refset
\Levy density estimation;
mean mixture curvature.

\section{Introduction}

We consider estimation and prediction for the Gamma model
$\mathrm{Ga}(\alpha, \beta)$
$(\alpha > 0,\ \beta > 0)$
whose density is given by
\begin{align}
p(x;\alpha,\beta)\frac{\dd x}{x}
=
\frac{1}{\Gamma(\alpha)}
\left( \frac{x}{\beta} \right)^{\alpha}
\exp\left(-\frac{x}{\beta}\right)
\frac{\dd x}{x}.
\label{Gammadensity}
\end{align}
We assume that the shape parameter $\alpha$ is known,
whereas the scale parameter $\beta$ is unknown.
The parameter $\alpha$ may be regarded as a role analogous to a sample size.

Estimation is regarded as a special case of prediction.
A general formulation of prediction problems is as follows.
Suppose that the simultaneous distribution of observables $x$ and $y$
belongs to a parametric family
\[
\{ p(x \mid \xi)p(y \mid \xi) \mid \xi \in \Xi \}.
\]
The parameter $\xi$ is unknown,
and the objective is to predict $y$ based on an observation $x$
by using a predictive density
$\hat{p}(y \mid x)$.

Many studies have shown that Bayesian predictive densities
\[
p_{\pi}(y \mid x)
=
\frac{
\int p(y \mid \xi)p(x \mid \xi)\pi(\xi)\,\dd\xi
}{
\int p(x \mid \bar{\xi})\pi(\bar{\xi})\,\dd\bar{\xi}
}
\]
based on a prior distribution $\pi(\xi)$
often outperform plug-in predictive densities
$p(y \mid \hat{\xi}(x))$,
where $\hat{\xi}(x)$ is an estimator of $\xi$
\citep{Aitchison75, AD75, Geisser93, K96}.

If we adopt a plug-in density
$p(y \mid \hat{\xi}(x))$
as a predictive density,
its loss can naturally be regarded as the loss of the estimator
$\hat{\xi}(x)$ itself.
In this sense, predictive distribution theory
provides a natural extension of estimation theory
under the Kullback--Leibler loss.

We evaluate the performance of a predictive density
$\hat{p}(y \mid x)$
by using the Kullback--Leibler divergence
\begin{align*}
D\{p(y \mid \xi), \hat{p}(y \mid x)\}
=
\int p(y \mid \xi)
\log
\frac{
p(y \mid \xi)
}{
\hat{p}(y \mid x)
}
\,\dd y.
\end{align*}
The corresponding risk function is defined by
\begin{align*}
R(\xi,\hat{p})
&=
\E\!\left[
D\{p(y \mid \xi),\hat{p}(y \mid x)\}
\,\middle|\, \xi
\right]
=
\int p(x \mid \xi)
\int p(y \mid \xi)
\log
\frac{
p(y \mid \xi)
}{
\hat{p}(y \mid x)
}
\,\dd y\,\dd x.
\end{align*}

A predictive density $p_1(y \mid x)$ is said to dominate
another predictive density $p_2(y \mid x)$
if it has no larger risk for any parameter value
and strictly smaller risk for some parameter value.
For example, in the normal model $\mathrm{N}(\mu,\sigma^2=1)$,
the Bayesian predictive density based on the Lebesgue prior (the Jeffreys prior) $\pi(\mu)=1$
dominates the plug-in predictive density based on the maximum likelihood estimator $\hat{\mu}=x$.
A survey of early results on predictive densities
is given in \cite{GLX12}.


The Gamma model \eqref{Gammadensity} forms a one-dimensional exponential family.
Indeed,
\begin{align*}
p(x \mid \beta)\,\dd x
=
\exp\left\{
\alpha
\left(
-\frac{1}{\beta}\frac{x}{\alpha}
-\log \beta
\right)
\right\}
\frac{x^{\alpha-1}}{\Gamma(\alpha)}
\,\dd x.
\end{align*}
The natural parameter is
$\theta=\frac{1}{\beta}$,
and the expectation parameter is
$\eta = \E_\beta [-x/\alpha ] = -\beta$.
The potential function is
$\psi(\theta)
=
-\log \theta
=
\log \beta.$


An estimator $\hat{\beta}(x)$ is said to be scale invariant if
$\hat{\beta}(cx)
=
c\hat{\beta}(x)$
$(c>0)$.
If we define
$b:=\hat{\beta}(1)$,
then scale invariance implies
$\hat{\beta}(x)
=
\hat{\beta}(x\cdot 1)
=
x\hat{\beta}(1)
=
bx$.
Therefore every scale invariant estimator is necessarily of the form
\[
\hat{\beta}_b(x)=bx,
\qquad b>0.
\]
The maximum likelihood estimator of $\beta$ is
$\hat{\beta}_{\mathrm{MLE}}
= x/\alpha$.

Let $y$ be distributed according to $\mathrm{Ga}(\alpha,\beta)$ and is independent of $x$.
The Kullback--Leibler divergence between
$p(y\mid \beta)$
and
$p(y\mid \hat{\beta})$
is
\begin{align*}
D (\beta , \, \hat{\beta})
&= \E_{y \mid \beta} \left[
\log \frac{
\displaystyle \frac{1}{\Gamma (\alpha)} \Bigl(\frac{y}{\beta} \Bigr)^\alpha \exp \Bigl(-\frac{y}{\beta} \Bigr)
}
{
\displaystyle \frac{1}{\Gamma (\alpha)} \Bigl(\frac{y}{\hat{\beta}} \Bigr)^\alpha \exp \Bigl(- \frac{y}{\hat{\beta}} \Bigr)
}
\right]
= \E_{y \mid \beta} \left[ \frac{y}{\hat{\beta}} - \frac{y}{\beta} + \alpha \log \Bigl(\frac{\hat{\beta}}{\beta} \Bigr) \right] 
= \alpha \left(\frac{\beta}{\hat{\beta}} - 1 - \log \frac{\beta}{\hat{\beta}} \right) .
\end{align*}

For the scale invariant estimator
\[
\hat{\beta}_b(x)=bx,
\]
we obtain
\begin{align*}
D(\beta,\hat{\beta}_b)
=
\alpha
\left(
\frac{\beta}{bx}
-
1
-
\log\frac{\beta}{bx}
\right).
\end{align*}
The corresponding risk function is
\begin{align*}
R(\beta,\hat{\beta}_b)
&=
\E_{x\mid\beta}
\left[
D(\beta,\hat{\beta}_b(x))
\right]
=
\E_{x\mid\beta}
\left[
\alpha
\left(
\frac{\beta}{bx}
-
1
-
\log\frac{\beta}{bx}
\right)
\right].
\end{align*}
Since
\[
\E_{x\mid\beta}\!\left[\frac{1}{x}\right]
=
\begin{cases}
\infty,
& 0<\alpha\le1,
\\[2mm]
\displaystyle
\frac{1}{(\alpha-1)\beta},
& \alpha>1,
\end{cases}
~~\mbox{and}~~~
\E_{x\mid\beta}
\left[
\log\frac{x}{\beta}
\right]
=
\frac{\Gamma'(\alpha)}{\Gamma(\alpha)}
=
\psi(\alpha),
\]
where $\psi$ denotes the digamma function,
we obtain
\begin{align*}
R(\beta,\hat{\beta}_b)
=
\begin{cases}
\infty,
& 0<\alpha\le1,
\\[2mm]
\displaystyle
\alpha
\left\{
\frac{1}{b(\alpha-1)}
-
1
+
\log b
-
\psi(\alpha)
\right\},
& \alpha>1.
\end{cases}
\end{align*}

For $\alpha>1$,
the risk is minimized at
$b=1/(\alpha-1)$,
and hence the best invariant estimator is
\begin{align}
\label{JE} 
\hat{\beta}_\mathrm{J}(x) = \hat{\beta}_{1 / (\alpha - 1)} (x) = \frac{x}{\alpha - 1},
\end{align}
which is also the posterior mean
with respect to the Jeffreys prior
\[
\pi_{\mathrm{J}}(\beta)\,\dd\beta
=
\frac{1}{\beta}\,\dd\beta.
\]
\cite{PN96} showed that the estimator \eqref{JE}
is admissible provided
$\alpha>1+\varepsilon$
for some $\varepsilon>0$.
On the other hand,
when $0<\alpha\le1$,
every scale invariant estimator has infinite risk.

Consider now a constant estimator
\[
\tilde{\beta}_c(x)=c,
\qquad c>0.
\]
Its risk is
\begin{align*}
R(\beta,\tilde{\beta}_c)
&=
D(\beta,\tilde{\beta}_c)
=
\alpha
\left(
\frac{\beta}{c}
-
1
-
\log\frac{\beta}{c}
\right)
<\infty.
\end{align*}
Hence,
for $0<\alpha\le1$,
every constant estimator dominates all scale invariant estimators.
Since constant estimators are trivially admissible,
this provides a simple example of a one-dimensional shrinkage phenomenon.

Next, consider the problem of prediction of $y$ independently
distributed according to $\mathrm{Ga}(T \alpha,\beta)$,
where $T>0$ is known, by using observation $x$.
Under parametric restrictions on the scale parameter,
the dominance and minimaxity properties of predictive densities
have been investigated in detail \cite{LMKS17}.
Despite these developments, the admissibility of Bayesian predictive densities appears to have remained unresolved.

Under the Jeffreys prior, the marginal density of $x$ is
\begin{align}
\label{marginalJ}
m_\mathrm{J} (x) = \int \frac{1}{\Gamma (\alpha)}
\left(\frac{x}{\beta} \right)^\alpha \exp \left(-\frac{x}{\beta} \right) \frac{\dd \beta}{\beta} 
= 1,
\end{align}
and the joint marginal density of $(x,y)$ is
\begin{align*}
m_{\mathrm{J}}(x,y)
&=
\int
\frac{1}{\Gamma(T \alpha)}
\left(\frac{y}{\beta}\right)^{T \alpha}
\exp\left(-\frac{y}{\beta}\right)
\frac{1}{\Gamma(\alpha)}
\left(\frac{x}{\beta}\right)^\alpha
\exp\left(-\frac{x}{\beta}\right)
\frac{\dd\beta}{\beta}
\\
&=
\frac{\Gamma(\alpha+T\alpha)}
{\Gamma(\alpha)\Gamma(T\alpha)}
\frac{x^\alpha y^{T\alpha}}
{(x+y)^{\alpha+T\alpha}}.
\end{align*}
The posterior density is
\begin{align}
\label{Jpost}
\pi_\mathrm{J} (\beta \mid x) \frac{\dd \beta}{\beta} = 
\frac{1}{\Gamma (\alpha)}
\left(\frac{x}{\beta} \right)^\alpha \exp \left(-\frac{x}{\beta} \right)
\frac{\dd \beta}{\beta},
\end{align}
and the Bayesian predictive density is
\begin{align}
p_{\mathrm{J}}(y\mid x) \frac{\dd y}{y}
&=
\frac{m_{\mathrm{J}}(x,y)}
{m_{\mathrm{J}}(x)}
\frac{\dd y}{y}
=
\frac{\Gamma(\alpha+T\alpha)}
{\Gamma(\alpha)\Gamma(T\alpha)}
\frac{x^\alpha y^{T\alpha}}
{(x+y)^{\alpha+T\alpha}}
\frac{\dd y}{y},
\label{Jpredictive}
\end{align}
which is an $F$-distribution.

Then, the Kullback--Leibler loss
\begin{align*}
\E_{y\mid\beta}
\Biggl[
\log
\frac{
\frac{1}{\Gamma(T\alpha)}
\bigl(\frac{y}{\beta}\bigr)^{T\alpha}
\exp(-y/\beta)
}{
\displaystyle
p_{\mathrm{J}}(y\mid x)
}
\Biggr]
=
\log
\frac{
\Gamma(\alpha)
}{
\Gamma(\alpha+T\alpha)x^\alpha\beta^{T\alpha}
}
-
\alpha
+
\E_{y\mid\beta}
\left[
(\alpha+T\alpha)\log(x+y)
\right],
\end{align*}
is finite and the corresponding risk is finite for all $\alpha>0$.
This is in sharp contrast to the estimation problem,
where every scale invariant estimator has infinite risk
when $0<\alpha\le1$.
The boundary $\alpha =1$
marks the onset of a breakdown of regular estimation structure,
whereas Bayesian prediction remains stable
for all $\alpha>0$.

In Section~2,
we prove
that the Bayes estimator based on the Jeffreys prior
is admissible for all $\alpha>1$ by removing the condition
``for some $\varepsilon>0$'' assumed in \cite{PN96}.
In Section~3,
we establish the main result of the paper:
the Bayesian predictive density
based on the Jeffreys prior
is admissible for all $\alpha>0$.
In Section~4,
we show that infinitesimal Bayesian prediction
corresponds not to the estimation of parameters,
but rather to the estimation of the \Levy density.
In Section 5,
we investigate,
from the viewpoint of information geometry,
the fact that prediction in the Gamma model
does not reduce to parameter estimation,
by evaluating the mean mixture curvature.

\section{Admissibility of estimators in the Gamma model}

In this section,
we prove directly, by using Blyth's method,
that the Bayes estimator based on the Jeffreys prior
is admissible under the Kullback--Leibler loss
for all $\alpha>1$.
This removes the assumption that $\alpha$ is separated from the boundary value $\alpha=1$,
which appeared in the form
$\alpha>1+\varepsilon$ for some $\varepsilon>0$
in Theorem 3.3 of \cite{PN96}.

When $0<\alpha\le1$,
the estimation problem undergoes a breakdown of the regular estimation structure.
In contrast,
as we shall see in the next section,
the corresponding prediction problem remains structurally stable.

\begin{theorem}
Suppose that $\alpha>1$.
Then the estimator
\begin{align}
\hat{\beta}_{\mathrm J}(x)
=
\frac{x}{\alpha-1}
\label{Jbetaest}
\end{align}
is admissible and minimax
under the Kullback--Leibler loss.
\end{theorem}

\vspace{0.5cm}

\noindent
{\it Proof.}
Throughout the proof,
we assume that $\alpha>1$.
Since $\hat{\beta}_{\mathrm J}$ has constant risk,
its admissibility implies minimaxity.
Hence, it suffices to prove admissibility.

Under the Kullback--Leibler loss,
the generalized Bayes estimator of an expectation parameter
in an exponential family is given by the posterior mean.
Hence, the generalized Bayes estimator of $\beta$
is the posterior mean of $\beta$.
The generalized Bayes estimator of $\theta:=\beta^{-1}$ with respect to the Jeffreys prior
$\beta^{-1}\,\dd\beta = \theta^{-1}\,\dd\theta$ is
\[
\hat{\theta}_{\mathrm J}
=
\frac{\alpha-1}{x},
\]
since the Kullback--Leibler loss is invariant
under the one-to-one transformation $\theta=\beta^{-1}$.
It suffices to prove the admissibility of
$\hat{\theta}_{\mathrm J}$ to prove the admissibility of
$\hat{\beta}_{\mathrm J}$.

To apply Blyth's method,
consider the following increasing sequence of finite but non-normalized priors
\begin{align}
g_l(\beta)\beta^{-1}\,\dd\beta
=
\tilde g_l(\theta)\theta^{-1}\,\dd\theta ~~~ (l=1,2,3,\ldots),
\qquad
\tilde g_l(\theta):=g_l(1/\theta),
\label{estpseq}
\end{align}
where
\begin{equation*}
\tilde g_l(\theta)
=
\frac12\{h_l(\theta)\}^2,
\end{equation*}
and
\begin{equation*}
h_l(\theta)
=
\begin{cases}
0,
& \theta\le1/l,
\\[2mm]
\displaystyle
1+\frac{\log\theta}{\log l},
& 1/l\le\theta\le1,
\\[4mm]
\displaystyle
1-\frac{\log\theta}{\log l},
& 1\le\theta\le l,
\\[4mm]
0,
& l\le\theta.
\end{cases}
\end{equation*}
The function $h_l$
is a modification of the function used by
\cite{BH82},
adapted so as to decay on both sides.

Let $\hat{\beta}_l$
denote the Bayes estimator of $\beta$
with respect to the prior
$g_l(\beta)\beta^{-1}\dd\beta$.
Then
\begin{align*}
\hat{\beta}_l
&=
\frac{
\int
\beta
\frac{1}{\Gamma(\alpha)}
\bigl(\frac{x}{\beta}\bigr)^\alpha
\exp \bigl( -\frac{x}{\beta} \bigr)
g_l(\beta)
\frac{\dd\beta}{\beta}
}{
\int
\frac{1}{\Gamma(\alpha)}
\bigl(\frac{x}{\beta}\bigr)^\alpha
\exp \bigl( - \frac{x}{\beta} \bigr)
g_l(\beta)
\frac{\dd\beta}{\beta}
}
=
\frac{
x
\int
(x\theta)^{\alpha-1}
\exp(-x\theta)
\tilde g_l(\theta)
\frac{\dd\theta}{\theta}
}{
\int
(x\theta)^\alpha
\exp(-x\theta)
\tilde g_l(\theta)
\frac{\dd\theta}{\theta}
}
\end{align*}
since the Bayes estimator is the posterior mean.
Thus, the Bayes estimator of $\theta$ is
\begin{align}
\label{ucl}
\hat{\theta}_l =
\frac{1}{\hat{\beta}_l} =
\frac{\frac{1}{x} \int (x \theta)^{\alpha} \exp (- x \theta) \tilde{g}_l (\theta) \frac{\dd \theta}{\theta}}
{\int (x \theta)^{\alpha - 1} \exp (- x \theta) \tilde{g}_l (\theta) \frac{\dd \theta}{\theta}}.
\end{align}
The numerator of \eqref{ucl} is
\begin{align*}
\int (x \theta)^{\alpha - 1} \exp (- x \theta) \tilde{g}_l (\theta) \dd \theta
&= \frac{\alpha - 1}{x} \int (x \theta)^{\alpha - 1} \exp (- x \theta) \tilde{g}_l (\theta) \frac{\dd \theta}{\theta}
+ \frac{1}{x^2} \int (x \theta)^{\alpha} \exp (- x \theta) {\tilde{g}'_l} (\theta) \frac{\dd \theta}{\theta},
\end{align*}
where integration by parts is used.
Since $h_l$ has compact support,
the boundary terms vanish.
Therefore,
\begin{align}
\hat{\theta}_l-\hat{\theta}_{\mathrm J}
=
\frac1{x^2}
\frac{
\int
(x\theta)^\alpha
\exp(-x\theta)
\tilde g_l'(\theta)
\frac{\dd\theta}{\theta}
}{
\int
(x\theta)^{\alpha-1}
\exp(-x\theta)
\tilde g_l(\theta)
\frac{\dd\theta}{\theta}
}.
\label{difference}
\end{align}

From
\begin{align*}
\frac{1}{\alpha}D(\theta,\hat{\theta})
=
\frac{\hat{\theta}}{\theta}
-
1
-
\log\frac{\hat{\theta}}{\theta},
\end{align*}
we obtain
\begin{align*}
&
\frac{\theta}{\alpha}
D(\theta,\hat{\theta}_{\mathrm J})
-
\frac{\theta}{\alpha}
D(\theta,\hat{\theta}_l)
=
\hat{\theta}_{\mathrm J}
-
\hat{\theta}_l
+
\theta
\log
\frac{\hat{\theta}_l}{\hat{\theta}_{\mathrm J}}
\le
\hat{\theta}_{\mathrm J}
-
\hat{\theta}_l
+
\theta
\left(
\frac{
\hat{\theta}_l-\hat{\theta}_{\mathrm J}
}{
\hat{\theta}_{\mathrm J}
}
\right),
\end{align*}
where the inequality
$\log t\le t-1$
is used.

The prior sequence \eqref{estpseq}
satisfies the assumptions of Blyth's method; see, e.g., \cite[p.~158]{Schervish95}.
We evaluate the difference between the Bayes risks of
$\hat{\theta}_\mathrm{J}$ and $\hat{\theta}_l$
with respect to this unnormalized prior sequence.
The Bayes risk difference is given by
\begin{align}
&\iint \frac{1}{\Gamma (\alpha)} (\theta x)^\alpha \exp (- \theta x) \frac{\alpha}{\theta} \frac{\theta}{\alpha}
\{D (\theta , \hat{\theta}_\mathrm{J}) - D (\theta , \hat{\theta}_l) \} \frac{\dd x}{x} \tilde{g}_l (\theta) \frac{\dd \theta}{\theta} \notag \\
&\leq \frac{\alpha}{\Gamma (\alpha)} \iint (\theta x)^\alpha \exp (- \theta x)
\left\{\hat{\theta}_\mathrm{J} - \hat{\theta}_l + \theta \left(\frac{\hat{\theta}_l - \hat{\theta}_\mathrm{J}}{\hat{\theta}_\mathrm{J}} \right) \right\}
\frac{\dd x}{x} \tilde{g}_l (\theta) \theta^{- 2} \dd \theta \notag \\
&= \frac{\alpha}{\Gamma (\alpha)} \iint (\theta x)^\alpha \exp (- \theta x)
\left(- 1 + \frac{\theta x}{\alpha - 1} \right) \tilde{g}_l (\theta) \theta^{- 2} \dd \theta
(\hat{\theta}_l - \hat{\theta}_\mathrm{J}) \frac{\dd x}{x} \notag \\
&= \frac{\alpha}{\Gamma (\alpha)} \int x^2 \int
\biggl\{\frac{1}{\alpha - 1} (\theta x)^{\alpha - 1} \tilde{g}_l (\theta) \exp (- \theta x)
 - (\theta x)^{\alpha - 2} \tilde{g}_l (\theta) \exp (- \theta x) \biggr\} \dd \theta
(\hat{\theta}_l - \hat{\theta}_\mathrm{J}) \frac{\dd x}{x}. \notag \\
&= \frac{\alpha}{\Gamma (\alpha)} \int x^2 \int
\biggl\{\frac{1}{\alpha - 1} (\theta x)^{\alpha - 1}
\tilde{g}_l (\theta) \exp (- \theta x)
+ \frac{1}{x} \frac{1}{\alpha - 1} (\theta x)^{\alpha - 1}
\tilde{g}'_l (\theta) \exp (- \theta x) \notag \\
& ~~~~~~~~~~~~~~~~~~~~~~~~~
+ \frac{1}{x} \frac{1}{\alpha - 1} (\theta x)^{\alpha - 1} \tilde{g}_l (\theta) (- x) \exp (- \theta x) \biggr\}
\dd \theta (\hat{\theta}_l - \hat{\theta}_\mathrm{J}) \frac{\dd x}{x} \notag \\
&= \frac{\alpha}{\Gamma (\alpha)} \int x^2 \int \frac{1}{x} \frac{1}{\alpha - 1} (\theta x)^{\alpha - 1}
\tilde{g}'_l (\theta) \exp (- \theta x) \dd \theta (\hat{\theta}_l - \hat{\theta}_\mathrm{J}) \frac{\dd x}{x}
\label{tochu0}
\end{align}
From \eqref{difference},
together with
\[
\tilde{g}_l (\theta) = \frac{1}{2} \{h_l (\theta) \}^2, \qquad
\tilde{g}'_l (\theta) = h_l (\theta) h'_l (\theta),
\]
the last expression in \eqref{tochu0} can be written as
\begin{align*}
&\frac{\alpha}{\Gamma (\alpha)}
\int \frac{1}{\alpha - 1} \int (\bar{\theta} x)^\alpha \exp (- \bar{\theta} x) \tilde{g}'_l (\bar{\theta})
\frac{\dd \bar{\theta}}{\bar{\theta}}
\frac{1}{x^2} \frac{\int (x \theta)^\alpha \exp (- x \theta) \tilde{g}'_l (\theta) \frac{\dd \theta}{\theta}}
{\int (x \theta)^{\alpha - 1} \exp (- x \theta) \tilde{g}_l (\theta) \frac{\dd \theta}{\theta}} \frac{\dd x}{x} \\
&= \frac{\alpha}{\Gamma (\alpha)} \int \frac{x^{-2}}{\alpha - 1}
\frac{\{\int(x\theta)^\alpha \exp (-x\theta) \tilde{g}'_l (\theta) \frac{\dd \theta}{\theta} \}^2}
{\int (x \theta)^{\alpha - 1} \exp (- x \theta) \tilde{g}_l (\theta) \frac{\dd \theta}{\theta}} \frac{\dd x}{x} \\
&= \frac{2 \alpha}{\Gamma (\alpha) (\alpha - 1)} \int x^{-2}
\frac{\Bigl[\int (x \theta)^{\alpha - 1} \exp (- x \theta)
\{(x \theta) h'_l (\theta) \} h_l (\theta) \frac{\dd \theta}{\theta} \Bigr]^2}
{\int (x \theta)^{\alpha - 1} \exp (- x \theta) \{h_l (\theta) \}^2 \frac{\dd \theta}{\theta}} \frac{\dd x}{x}.
\end{align*}
By the Cauchy--Schwarz inequality,
\begin{align*}
\frac{2 \alpha}{\Gamma (\alpha) (\alpha - 1)} & \int x^{-2}
\frac{\Bigl[\int (x \theta)^{\alpha - 1} \exp (- x \theta) \{(x \theta) h'_l (\theta) \}
h_l (\theta) \frac{\dd \theta}{\theta} \Bigr]^2}
{\int (x \theta)^{\alpha - 1} \exp (- x \theta) \{h_l (\theta) \}^2 \frac{\dd \theta}{\theta}} \frac{\dd x}{x} \\
\leq&\ \frac{2 \alpha}{\Gamma (\alpha) (\alpha - 1)} \int x^{-2}
\int (x \theta)^{\alpha - 1} \exp (- x \theta) \{(x \theta) h'_l (\theta) \}^2 \frac{\dd \theta}{\theta} \frac{\dd x}{x} \\
=&\ \frac{2 \alpha}{\Gamma (\alpha) (\alpha - 1)} \int \theta^{2} 
\int
(x \theta)^{\alpha - 1} \exp (- x \theta) \{h'_l (\theta) \}^2 \frac{\dd x}{x} \frac{\dd \theta}{\theta} \\
=&\ \frac{2 \alpha}{\Gamma (\alpha) (\alpha - 1)} \int \theta^2 \Gamma (\alpha - 1) \{h'_l (\theta) \}^2 \frac{\dd \theta}{\theta} \\
=&\ \frac{2 \alpha}{(\alpha - 1)^2} \int^\infty_0 \theta \{h'_l (\theta) \}^2 \dd \theta.
\end{align*}
Since
\begin{equation*}
h'_l (\theta) =
\begin{cases} 0 & (\theta \leq \frac{1}{l}) \\
\frac{1}{\log l}\frac{1}{\theta} & (\frac{1}{l} \leq \theta \leq 1) \\
- \frac{1}{\log l}\frac{1}{\theta} & (1 \leq \theta \leq l) \\
0 & (l \leq \theta),
\end{cases}
\end{equation*}
we have
\begin{align*}
\int^\infty_ 0 \theta \{h'_l (\theta) \}^2 \dd \theta
= & \frac{1}{(\log l)^2} \int^1_{\frac{1}{l}} \theta \frac{1}{\theta^2} \dd \theta
+ \frac{1}{(\log l)^2} \int^l_1 \theta \frac{1}{\theta^2} \dd \theta
=
 \frac{1}{(\log l)^2} [ \log \theta ]^l_{\frac{1}{l}}
=
 \frac{2}{\log l}.
\end{align*}
Hence, the last expression converges to $0$ as $l \to \infty$.
Therefore, the estimator \eqref{Jbetaest}
is admissible.
\hfill $\Box$

\section{Admissibility of Bayesian Prediction in the Gamma Model}
In this section,
we show that, in contrast to the estimation problem,
the predictive density based on the Jeffreys prior
is admissible for all $\alpha>0$.
Thus, Bayesian prediction does not exhibit the breakdown of regular estimation structure
that occurs when $0<\alpha\le1$.

Suppose that we observe a random variable $x$
and predict another random variable $y$,
where the distributions of $x$ and $y$
given an unknown parameter $\beta >0$
are
$\mathrm{Ga} (\alpha, \beta)$ and
$\mathrm{Ga} (\alpha T, \beta)$, respectively,
and $x$ and $y$ are independent given $\beta$.
Here,
$\alpha >0$ and $T>0$ are known constants.
Their densities are given by
\begin{align*}
p(x \mid \beta) \frac{\dd x}{x}
=&\ \frac{1}{\Gamma (\alpha)}
\left(\frac{x}{\beta} \right)^{\alpha}
\exp \left(- \frac{x}{\beta} \right)
\frac{\dd x}{x}, \\
p(y \mid \beta) \frac{\dd y}{y}
=&\ \frac{1}{\Gamma (\alpha T)}
\left(\frac{y}{\beta} \right)^{\alpha T}
\exp \left(- \frac{y}{\beta} \right)
\frac{\dd y}{y}.
\end{align*}

Define
\[
u = x + y, \mbox{~~and~~~} v = \frac{y}{x + y}.
\]
Then
\begin{align*}
p_\beta (x , y) & \frac{\dd x}{x} \frac{\dd y}{y}
= \frac{1}{\Gamma (\alpha)} \left(\frac{x}{\beta} \right)^\alpha
\exp \left(- \frac{x}{\beta} \right)
\frac{1}{\Gamma (\alpha T)} \left(\frac {y}{\beta} \right)^{\alpha T}
\exp \left(- \frac{y}{\beta} \right)
\frac{\dd x}{x} \frac{\dd y}{y} \\
&= \frac{\Gamma (\alpha + \alpha T)}{\Gamma (\alpha) \Gamma (\alpha T)}
v^{{\alpha T}-1} (1 - v)^{\alpha - 1} \dd v 
\frac{1}{\Gamma (\alpha + \alpha T)}
\left(\frac{u}{\beta} \right)^{\alpha + {\alpha T}}
\exp \left(- \frac{u}{\beta} \right)
\frac{\dd u}{u},
\end{align*}
and the distribution of $v$ is a beta distribution $\mathrm{Be}(\alpha T, \alpha)$ and
does not depend on $\beta$.
Hence, by sufficiency reduction,
\begin{gather*}
p_\beta (x , y) \frac{\dd x}{x} \frac{\dd y}{y}
= p_\beta (u , v) \frac{\dd u}{u} \dd v
= p_\beta (u) \frac{\dd u}{u} p(v) \dd v,
\end{gather*}
and
\begin{gather}
\log \frac{p_\beta (y \mid x)}{p_\pi (y \mid x)}
= \log \frac{p_\beta (u)}{m_\pi (u)} - \log \frac{p_\beta (x)}{m_\pi (x)}.
\label{sreduction}
\end{gather}

Consider the Gamma process $(x_t)_{t\ge0}$, that is, a process with independent increments
such that the distribution of $x_t$ is $\mathrm{Ga}(\alpha t,\beta)$ for $t\ge0$.

The original prediction problem is equivalent to predicting
\[
 y = x_{1+T} - x_1
\]
which is distributed according to $\mathrm{Ga}(\alpha T,\beta)$,
by using observation $x_1$ distributed according to $\mathrm{Ga}(\alpha,\beta)$.

From \eqref{sreduction}, we have
\begin{align}
 \E_{x , y \mid \beta} \left[ \log \frac{p_\beta (y \mid x)}{p_\pi (y \mid x)} \right]
&= \E_{x_{1+T} \mid \beta} \left[ \log \frac{p_\beta (x_{1+T})}{m_\pi (x_{1+T})} \right]
- \E_{x_1 \mid \beta} \left[ \log \frac{p_\beta (x_1)}{m_\pi (x_1)} \right] \notag \\
&= \int^{1 + T}_{1} \frac{\dd}{\dd t} \E_{x_t \mid \beta , t}
\left[ \log \frac{p_{t, \beta} (x_t)}{m_{t , \pi} (x_t)} \right] \dd t.
\label{integralrepresentation}
\end{align}
Therefore,
the problem of predicting $y$ from an observation of $x$
is equivalent to the problem of predicting
$x_t$ $(1 < t < 1+T)$
from observations of the Gamma process
$x_t$ $(0 \leq t \leq 1)$.

We call
\begin{align}
\frac{\dd}{\dd t} \E_{x_t \mid \beta, t} \left[ \log \frac{p_{t, \beta} (x_t)}{m_{t, \pi} (x_t)} \right]
\label{ipr}
\end{align}
the infinitesimal prediction risk.

Using this integral representation,
we are now ready to establish the main result of the paper.

\begin{theorem}
The Bayesian predictive density
\begin{align*}
p_{\mathrm{J}}(y\mid x)
&=
\frac{\Gamma(\alpha+T\alpha)}
{\Gamma(\alpha)\Gamma(T\alpha)}
\frac{x^\alpha y^{T\alpha}}
{(x+y)^{\alpha+T\alpha}}
\end{align*}
based on the Jeffreys prior
$\pi_\mathrm{J}(\beta)$
is admissible under the Kullback--Leibler loss
for all $\alpha >0$ and $T > 0$.
\end{theorem}

\begin{proof}
Consider the increasing sequence of finite unnormalized priors
\[
f_{l}(\beta) \frac{\dd \beta}{\beta}
\coloneqq \exp\biggl\{-\biggl(\frac{\log \beta}{l} \biggr)^2 \biggr\} \frac{\dd \beta}{\beta}
\qquad (l = 1,2,3,\ldots),
\]
which satisfies the assumptions of Blyth's method.
Observe that, up to a normalizing constant,
$f_l(\beta)\,\dd\beta/\beta$
is a log--normal distribution whose variance on the logarithmic scale tends to infinity as $l\to\infty$.
Since the Kullback--Leibler loss is convex as a function of the predictive density,
the standard Blyth argument implies the admissibility of the Jeffreys predictive density;
see, for example, \cite[p.~158]{Schervish95}.

Let $p_l(y \mid x)$ denote the Bayes predictive density
based on
$f_{l}(\beta) \frac{\dd \beta}{\beta}$.
Then,
from \eqref{integralrepresentation},
\begin{align}
& \int \E_{x , y \mid \beta} \left[ \log \frac{p_\beta (y \mid x)}{p_\mathrm{J} (y \mid x)} \right]
f_l(\beta) \frac{\dd \beta}{\beta}
-
\int \E_{x , y \mid \beta} \left[ \log \frac{p_\beta (y \mid x)}{p_l (y \mid x)} \right]
f_l(\beta) \frac{\dd \beta}{\beta} \notag \\
&= \int \int^{1 + T}_{1} \frac{\dd}{\dd t} \E_{x_t \mid \beta , t}
\left[ \log \frac{m_{t,l}(x_t)}{m_{t,\mathrm{J}}(x_t)} \right] \dd t
f_l(\beta) \frac{\dd \beta}{\beta} \notag \\
&= \int \int^{1 + T}_{1} \frac{\dd}{\dd t} \E_{x_t \mid \beta , t}
\Bigl[ \log m_{t,l}(x_t) \Bigr] \dd t
f_l(\beta) \frac{\dd \beta}{\beta} \notag \\
&= \int^{1 + T}_{1}
\partial_{t} \int m_{t,l}(z) \log m_{{t},l}(z)
\frac{\dd z}{z} \dd t,
\label{Briskdiff}
\end{align}
where $x_t$ is denoted as $z$,
\[
m_{t,\mathrm{J}}(z)
\coloneqq
\int p_{t}(z \mid \beta) \frac{\dd \beta}{\beta} = 1
\mbox{~~and~~}
m_{t,l}(z)
\coloneqq
\int p_{t}(z \mid \beta) f_{l}(\beta) \frac{\dd \beta}{\beta}.
\]
To verify the final condition of Blyth's method,
it remains to show that
\eqref{Briskdiff} converges to $0$
as $l \to \infty$.
This follows from the Lemma \ref{lemma1} below.
\end{proof}

\vspace{0.5cm}
\begin{lemma}
\label{lemma1}
For every finite interval $[a,b] \subset \mathbb{R}_+$,
there exists a constant $C>0$
such that
for all $t \in [a,b]$,
\begin{align}
\biggl| \partial _{t} \int m_{t,l}(z) \log m_{{t},l}(z) \frac{\dd z}{z} 
\biggr| 
\leq \frac{C}{l}.
\label{lemma0eq}
\end{align}
\end{lemma}

\vspace{0.5cm}

\begin{proof}
Let
\begin{align*}
\overline{f}_{l}(\beta) \frac{\dd \beta}{\beta}
&\coloneqq \frac{1}{\sqrt{\pi}\, l} f_l (\beta) \frac{\dd \beta}{\beta}
~~~ \mathrm{and} ~~~
\overline{m}_{t, l}(z) \frac{\dd z}{z}
\coloneqq \frac{1}{\sqrt{\pi}\,l} m_{t,l} (z)
\frac{\dd z}{z}.
\end{align*}
Then,
$\overline{f}_{l}(\beta) \frac{\dd \beta}{\beta}$
and
$\overline{m}_{t, l}(z) \frac{\dd z}{z}$
are probability measures.

Let
\[
\xi \coloneqq \frac{\log \beta}{l}.
\]
When $z$ and $\beta$ are jointly distributed according to
\[
p_t(z \mid \beta) \overline{f}_{l}(\beta) 
\frac{\dd z}{z}
\frac{\dd \beta}{\beta},
\]
the probability density of $\xi$ is
\begin{align}
\phi(\xi) \coloneqq \frac{1}{\sqrt{\pi}} \exp(-\xi^2). 
\label{phi}
\end{align}

Let
\[
u = \log \frac{z}{\beta}.
\]
Then the probability density of $u$ is
\begin{align*}
 r_t(u) \dd u \coloneqq \frac{1}{\Gamma(\alpha t)} \ee^{\alpha t u} \exp(-\ee^{u}) \dd u,
\end{align*}
which is the log--Gamma distribution with shape parameter $\alpha t$,
and
\begin{align*}
 \E[u] = \psi(\alpha t) \coloneqq
 \frac{\Gamma' (\alpha t)}{\Gamma (\alpha t)}.
\end{align*}

Now define
\begin{align*}
 v \coloneqq \frac{\log z-\psi(\alpha t)}{l}
=\frac{u + \log \beta - \psi(\alpha t)}{l}
= \frac{u-\psi(\alpha t)}{l}+\xi.
\end{align*}
Then, the marginal density of $v$ is
\begin{align}
\label{qdensity}
q_{t, l}(v)
=\int \phi \Bigl(v - \frac{u-\psi(\alpha t)}{l} \Bigr) r_t(u) \dd u.
\end{align}

For fixed $l$,
$z$ and $v$ are in one-to-one correspondence,
and so are $\beta$ and $\xi$.
Since
\[
l \dd v=\frac{\dd z}{z},
\]
we have
\[
\overline{m}_{t,l}(z) \frac{\dd z}{z}
= q_{t,l}(v) \dd v
= q_{t,l}(v) \frac{1}{l} \frac{\dd z}{z}.
\]

Therefore,
\begin{align}
\partial_{t} & \int m_{t,l}(z) \log m_{t, l}(z) \frac{\dd z}{z}
= \sqrt{\pi} l \partial_{t} \int \overline{m}_{t, l}(z) \log \overline{m}_{t,l}(z) \frac{\dd z}{z} \notag \\
&= \sqrt{\pi} l \, \partial_{t} \Bigl\{\int q_{t, l}(v) \log q_{t, l}(v) \dd v-\log l \Bigr\}
=\sqrt{\pi}l  \, \partial_{t} \int q_{t, l}(v) \log q_{t, l}(v) \dd v  \notag \\
&=\sqrt{\pi} l  \, \int \partial_{t} q_{t, l}(v) \log q_{t, l}(v) \dd v.
\label{entropy}
\end{align}
Hence,
to prove \eqref{lemma0eq},
it suffices to show
\begin{align}
\sqrt{\pi} l \, \biggl| \int \partial_{t} q_{t, l}(v) \log q_{t, l}(v) \dd v \biggr| \leq \frac{C}{l}.
\label{entropy2}
\end{align}

We first show that there exists a constant $C_1>0$,
independent of $l$,
such that
\begin{align}
\bigl| \log q_{t, l}(v) \bigr| \leq C_1 (1+v^2).
\label{logup}
\end{align}

Let
$\tilde{u} \coloneqq u-\psi(\alpha t)$
and
$\tilde{r}_t (\tilde{u}) \coloneqq r_t(\tilde{u} + \psi(\alpha t))$.
If $R$ is sufficiently large,
then
\[
\int^R_{-R} \tilde{r}_t(\tilde{u}) \dd \tilde{u} \geq \frac{1}{2}.
\]
Hence, from \eqref{qdensity},
\begin{align*}
q_{t, l}(v)
&\geq
\int ^R_{-R} \phi \Bigl(v-\frac{\tilde{u}}{l} \Bigr) \tilde{r}_t(\tilde{u}) \dd \tilde{u}
\geq \frac{1}{2} \inf_{|\tilde{u}| \leq R} \phi \Bigl(v-\frac{\tilde{u}}{l} \Bigr)
\geq \frac{1}{2} \inf_{|\tilde{u}| \leq R} \phi (v-\tilde{u}).
\end{align*}
Since \eqref{phi} and
\[
(v-\tilde{u})^2 \leq 2(v^2 + {\tilde{u}}^2) \leq 2(v^2 + R^2)
\]
for $|\tilde{u}| \leq R$,
there exists a constant $C_2>0$,
independent of $l$,
such that
\begin{align*}
q_{t, l}(v)
\geq 
\frac{1}{2} \inf_{|\tilde{u}| \leq R} \phi (v-\tilde{u})
\geq C_2 \ee^{-2 v^2}.
\end{align*}
From
\[
\sup_v \phi(v) = \frac{1}{\sqrt{\pi}},
\]
and \eqref{qdensity},
we obtain
\begin{align*}
\log C_2 - 2 v^2
\leq
\log q_{t, l}(v)
\leq
-\frac{1}{2} \log \pi.
\end{align*}
Thus \eqref{logup} follows.

Therefore,
to prove \eqref{entropy2},
it suffices to show
\[
\int (1+v^2) \left| \partial_{t} q_{t, l}(v) \right| \dd v
\leq \frac{C}{l^2}.
\]

From \eqref{qdensity} and
\[
\partial_{t} r_{t} (u)= \alpha (u-\psi(\alpha t)) r_{t}(u),
\]
we obtain
\begin{align}
\partial_{t} q_{t, l}(v)
&=  \alpha \int \frac{\psi'(\alpha t)}{l} \phi'
 \Bigl(v-\frac{u}{l}+\frac{\psi(\alpha t)}{l} \Bigr)
 r_{t}(u)\dd u 
+
 \alpha \int \phi \Bigl(v-\frac{u}{l}+\frac{\psi(\alpha t)}{l} \Bigr) (u-\psi (\alpha t)) r_{t}(u)\dd u \notag \\
&=  \alpha \frac{\psi'(\alpha t)}{l} \int \phi'\Bigl(v-\frac{\tilde{u}}{l} \Bigr) \tilde{r}_{t}(\tilde{u}) \dd \tilde{u}
+ \alpha \int \phi \Bigl(v-\frac{\tilde{u}}{l} \Bigr) \tilde{u}\tilde{r}_{t}(\tilde{u}) \dd \tilde{u}.
\label{partialq}
\end{align}

For a twice continuously differentiable function $g(v)$,
\begin{align}
 g(v +h)
 = g(v) + h \int_0^1 g'(v+h s) \dd s 
 = g(v) + h g'(v)
 + h^2 \int^1_0 (1-s) g''(v+h s)\dd s.
\label{C^2}
\end{align}
From \eqref{partialq}, \eqref{C^2}, $\E [\tilde{u}] = \int \tilde{u} \tilde{r}_{t}(\tilde{u})\dd \tilde{u} = 0$ and
$\E [\tilde{u}^2] = \int \tilde{u}^2 \tilde{r}_{t}(\tilde{u})\dd \tilde{u} = \psi'(\alpha t)$,
we have
\begin{align}
\frac{1}{\alpha} \partial_{t} q_{t, l}(v)
&=
\frac{\psi'(\alpha t)}{l} \int \biggl\{ \phi'(v) -
\frac{\tilde{u}}{l} \int^1_0 \phi'' \Bigl(v-\frac{\tilde{u}}{l}s \Bigr) \dd s 
\biggr\} \tilde{r}_{t}(\tilde{u})\dd \tilde{u} \notag \\
&~~+\int \Bigl\{\phi(v)-\frac{\tilde{u}}{l} \phi'(v)
+ \frac{\tilde{u}^2}{l^2} \int^1_0 (1-s) \phi'' \Bigl(v-\frac{\tilde{u}}{l}s \Bigr) \dd s
\Bigr\} \tilde{u} \tilde{r}_{t}(\tilde{u}) \dd \tilde{u} \notag \\
&=\frac{\psi'(\alpha t)}{l}\phi'(v)-\frac{\psi'(\alpha t)}{l^2}
\int \int^1_0 \tilde{u} \phi''
\Bigl(v-\frac{\tilde{u}}{l} s \Bigr) \dd s \tilde{r}_{t}(\tilde{u}) \dd \tilde{u} \notag \\
&~~~ - \frac{\psi'(\alpha t)}{l} \phi'(v)
+\int \frac{\tilde{u}^3}{l^2} \int^1_0(1-s) \phi''
\Bigl(v - \frac{\tilde{u}}{l}s \Bigr) \dd s
\tilde{r}_{t}(\tilde{u}) \dd \tilde{u} \notag \\
&=\frac{1}{l^2} \int
\int^1_0 \bigl\{(1-s)\tilde{u}^3-\psi'(\alpha t)\tilde{u} \bigr\}
\phi'' \Bigl(v-\frac{\tilde{u}}{l} s \Bigr) \dd s
\tilde{r}_{t}(\tilde{u}) \dd \tilde{u}.
\label{dq}
\end{align}

Since
\[
\phi''(v)=(4v^2-2)\phi(v),
\]
there exists a constant $C_3 > 0$ such that
\begin{align}
\int(1+ v^2) \bigl| \phi''(v-w) \bigr| \dd v
&= \int \bigl\{1+ (v+w)^2 \bigr\} \bigl| \phi''(v) \bigr| \dd v
\leq \int(1+2v^2+2 w^2) \bigl| \phi''(v) \bigr| \dd v \notag \\
&= \int \bigl| (1+2 v^2+2 w^2)(4 v^2-2) \bigr| \phi (v) \dd v
\leq C_3 (w^2+1).
\label{1+v^2}
\end{align}

From \eqref{logup}, \eqref{dq}, and \eqref{1+v^2}, we have
\begin{align}
\frac{1}{\alpha} \int \bigl| \partial_{t} q_{t, l}(v) \bigr| \, \bigl| \log q_{t, l}(v) \bigr| \dd v
&\leq 
\frac{C_1}{l^2} \int(1+v^2) \biggl| \int
\int^1_0 \bigl\{(1-s)\tilde{u}^3-\psi'(\alpha t)\tilde{u} \bigr\}
\phi'' \Bigl(v-\frac{\tilde{u}}{l} s \Bigr) \dd s
\tilde{r}_{t}(\tilde{u}) \dd \tilde{u}\biggr| \dd v \notag \\
&\leq \frac{C_1}{l^2}
\int \Bigl(|\tilde{u}|^3 + \psi'(\at) |\tilde{u}| \Bigr)
\int^1_0 \int(1+v^2)
\biggl| \phi'' \Bigl(v-\frac{\tilde{u}}{l} s \Bigr) \biggr| \dd v \dd s \tilde{r}_{t}(\tilde{u}) \dd \tilde{u} \notag \\
&\leq \frac{C_1 C_3}{l^2} \int
\Bigl(|\tilde{u}|^3 + \psi'(\alpha t) |\tilde{u}| \Bigr)
\Bigl(1+ \frac{\tilde{u}^2}{l^2} \Bigr)
\tilde{r}_{t}(\tilde{u})
\dd \tilde{u}
\leq \frac{C_4}{l^2}.
\label{entineq}
\end{align}
Since the third and fifth absolute moments of $\tilde r_t$
are finite and continuous in $t$,
the constant $C_4$ can be chosen uniformly over
$t\in[a,b]$.
Therefore,
\eqref{entropy2} follows from \eqref{entineq}.
\end{proof}

\section{Infinitesimal prediction and estimation of \Levy measures}

The Kullback--Leibler risk of Bayesian predictive densities
can often be represented as an integral of the Kullback--Leibler risk
of the corresponding Bayes estimators for unknown parameters.
See \cite{K:JMVA2006} for the Poisson models
and \cite{BGX:AS08} for the normal models.
For a class of Poisson regression models, the Kullback-Leibler risk for Bayesian predictive densities
is represented as an integral of the Kullback-Leibler risk
for the corresponding Bayes extended estimators \citep{K:JJSD2024}.

Unlike the normal and Poisson models, infinitesimal prediction in the Gamma model
does not reduce to parameter estimation.
Instead, it reduces to the estimation of the \Levy measure.
Related problems in information theory have been studied
in connection with channel capacities of normal, Poisson,
and \Levy channels \citep{GuoShamaiVerdu13, JVW17}.

In this section,
we discuss the relation between infinitesimal prediction
and the estimation of \Levy measures.
We consider \Levy measures of the class of pure jump subordinators including the Gamma process,
and the discussion becomes simpler than those for general \Levy processes.

We first summarize several basic facts on \Levy processes and pure jump subordinators.
Detailed expositions can be found in
\cite{Applebaum09}, \cite{ContTankov04}, and \cite{{Sato13}}.

A \Levy\ process $(y_t)_{t\ge0}$ is a cadlag stochastic process
with $y_0=0$ almost surely,
independent and stationary increments,
and continuity in probability.
Every \Levy process is infinitely divisible, and conversely,
every infinitely divisible distribution
is realized as the distribution of $x_1$
for some \Levy process.

A subordinator
is a one-dimensional nondecreasing \Levy process.
The characteristic function of a subordinator $\{y_t\}$ is represented as
\[
\E[\exp(\ii u y_t)]
=
\exp\left\{
\ii bu
+
\int_0^\infty
(e^{\ii uy}-1)\,
\tilde p(\dd y)
\right\},
\]
where $b\ge0$ and
$\tilde p$
is a \Levy measure satisfying
\[
\int_0^\infty (y\wedge1)\,\tilde p(\dd y)<\infty.
\]
In the following,
we assume $b=0$ and there exists \Levy density i.e. $\tilde{p}(\dd y) = \tilde{p}(y) \dd y$.
Such subordinators are called pure jump subordinators
with \Levy densities,
which is a natural class that includes the Gamma process.

The Kullback--Leibler divergence between two \Levy densities of pure jump subordinators
is defined as follows.
This definition is based on the relative entropy between \Levy processes;
see, e.g., \cite[p.~312]{ContTankov04}.

\begin{definition}
We define the Kullback--Leibler divergence between the two \Levy densities $\tilde{p}(y)$ and $\tilde{q}(y)$
corresponding to pure jump subordinators by
\[
 D(\tilde{p},\tilde{q}) =
\int \tilde{q}(y) - \tilde{p}(y) - \tilde{p}(y) \log \frac{\tilde{q}(y)}{\tilde{p}(y)} \dd y.
\]
\end{definition}

\vspace{0.5cm}

For the estimation of \Levy measures,
the following result holds,
which corresponds to the well-known result of
\cite{Aitchison75}
for the estimation of probability densities.

\vspace{0.5cm}

\begin{theorem}
\label{LevyAitchson}
Let $\tilde p_\xi(y)\,\dd y$ $(y\ge0)$
be the \Levy measure of a pure jump subordinator
indexed by a parameter $\xi$,
and let $x$ be an independent observation whose distribution is specified by $\xi$.
Consider the problem of estimating the \Levy density
$\tilde p_\xi(y)$
based on the observation $x$.

Suppose that an estimator of the \Levy density,
denoted by $\tilde q(y;x)$,
is evaluated by the Bayes risk
\begin{align*}
&\iint
\left\{
\tilde q(y;x)
-
\tilde p_\xi(y)
-
\tilde p_\xi(y)
\log
\frac{\tilde q(y;x)}
{\tilde p_\xi(y)}
\right\}
\,dy\,
p(x\mid\xi)\,\dd x\,
\pi(\xi)\,\dd\xi ,
\end{align*}
with respect to a proper prior
$\pi(\xi)\,\dd\xi$,
and assume that the Bayes risk is finite.

Then the Bayes estimator minimizing the Bayes risk
is given by
\[
\tilde q(y;x)
=
\tilde p_\pi(y\mid x)
:=
\int
\tilde p_\xi(y)\,
\pi(\xi \mid x)\,
\dd\xi,
\]
namely, the posterior mean of the \Levy density.
\end{theorem}

\vspace{0.5cm}

We call $\tilde{p}_\pi (y \mid x)$ the Bayesian predictive \Levy density.

\vspace{0.5cm}

\begin{proof}[Proof of Theorem \ref{LevyAitchson}]
We have
\begin{align*}
&\iint
\left\{
\tilde q(y;x)
-
\tilde p_\xi(y)
-
\tilde p_\xi(y)
\log
\frac{\tilde q(y;x)}
{\tilde p_\xi(y)}
\right\}
dy\,
p(x\mid\xi)\,\dd x\,
\pi(\xi)\,\dd\xi
\\
&=
\int
p_\pi(x)
\Biggl[
\iint
\Bigl\{
\tilde q(y;x)
-
\tilde p_\xi(y)
-
\tilde p_\xi(y)\log \tilde q(y;x)
+
\tilde p_\xi(y)\log \tilde p_\xi(y)
\Bigr\}
dy\,
\pi(\xi\mid x)\,\dd\xi
\Biggr]
\dd x
\\
&=
\int
p_\pi(x)
\int
\left\{
\tilde q(y;x)
-
\tilde p_\pi(y\mid x)
-
\tilde p_\pi(y\mid x)
\log
\frac{\tilde q(y;x)}
{\tilde p_\pi(y\mid x)}
\right\}
dy\,
\dd x
\\
&\qquad
-
\int
\tilde p_\pi(y\mid x)
\log \tilde p_\pi(y\mid x)
\,dy\,
p_\pi(x)\,\dd x
+
\iint
\tilde p_\xi(y)
\log \tilde p_\xi(y)
\,dy\,
\pi(\xi)\,\dd\xi .
\end{align*}
The second and third terms do not depend on
$\tilde q(y;x)$.
The first term is minimized when
\[
\tilde q(y;x)
=
\tilde p_\pi(y\mid x).
\]
Therefore, the Bayes risk is minimized by
$\tilde q(y;x)=\tilde p_\pi(y\mid x)$.
\end{proof}

We consider pure jump subordinator models
based on the Esscher transform 
(exponential tilting) including the Gamma model.
For a \Levy density
$\tilde p(y)\,\dd y$
on $\mathbb R_+$,
if
\[
\int_1^\infty
e^{\theta y}\tilde p(y)\,\dd y < \infty,
\]
then
\[
\tilde p_\theta(y)\,\dd y
:=
e^{\theta y}\tilde p(y)\,\dd y
\]
is again a \Levy measure.
If $(x_t)_{t\ge0}$ is a pure jump process
generated by a \Levy measure with unknown parameter $\theta$ for Esscher transfom,
then the terminal value $x_t$ is a sufficient statistic
for the observation
$\{x_\tau: 0\le \tau \le t \}$;
see, for example,
Küchler and Sørensen (1997).

Consider the problem of predicting
\[
 y_s \coloneqq x_{t+s} - x_t \sim p_s(y \mid \zeta),
\]
where $\{x_t\}$ is a pure jump subordinator, 
by using observation \(x_t \sim p_t(x \mid \zeta) \).

Under suitable regularity conditions,
the relation
\[
 \tilde{p}(y)\,\dd y
 =
 \lim_{s\downarrow 0}
 \frac{1}{s}p_s(y)
\]
holds, where $p_s(y)$ denotes the probability density of
$y_s=x_{t+s}-x_t$.
The condition is fulfilled in many commonly used models,
including the Gamma process;
see \cite{BarndorffNielsen00,RuschendorfWoerner02}.

\vspace{0.5cm}

\begin{example} Gamma processes

The probability density of $y_s$ is
\[
 p_{s,\theta}(y) = \frac{1}{\Gamma(\alpha s)} \frac{(\theta y)^{\alpha s}}{y} \exp(-\theta y),
\]
where $\theta = 1/\beta$ is the unknown parameter.
The L\'{e}vy measure is
\begin{align*}
\tilde{p}(y) \dd y &= \lim_{s \downarrow 0} \frac{1}{s} \frac{1}{\Gamma(\alpha s)} \frac{(\theta y)^{\alpha s}}{y} \exp(-\theta y)
= \frac{\alpha}{y} \exp(-\theta y).
\end{align*}
Thus, the parameter of the Gamma \Levy densities corresponds to the parameter of the Esscher transform.
When $x$ is distributed according to $\mathrm{Ga}(\alpha,\beta)$, 
the posterior density $\pi_\mathrm{J} (\beta \mid x)$ is \eqref{Jpost}.
Then, the Bayesian predictive \Levy density is
\begin{align}
\int \frac{\alpha}{y} \exp \biggl( - \frac{y}{\beta} \biggr)
\frac{1}{\Gamma(\alpha)}
\biggl( \frac{x}{\beta} \biggr)^\alpha
\exp \biggl( - \frac{x}{\beta} \biggr) \frac{\dd \beta}{\beta} \dd y
&=
\alpha
\biggl( \frac{x}{x+y} \biggr)^\alpha \frac{\dd y}{y}.
\label{BPL}
\end{align}
The Bayesian predictive density $p_{s,\mathrm{J}} (y \mid x)$ for $y_s$ is \eqref{Jpredictive}.
Then, \eqref{BPL} coincides with
\[
\lim_{s \downarrow 0} \frac{1}{s} p_{s,\mathrm{J}} (y \mid x) \frac{\dd y}{y}
= \lim_{s \downarrow 0} \frac{\alpha}{\alpha s}
\frac{\Gamma (\alpha + \alpha s)}{\Gamma (\alpha s) \Gamma (\alpha)}
\frac{x^{\alpha} y^{\alpha s}}{(x + y)^{\alpha + \alpha s}}  \frac{\dd y}{y}
= 
\alpha
\biggl( \frac{x}{x+y} \biggr)^\alpha
\frac{\dd y}{y}.
\]
%

Although $p_{s,\theta}(y)$ is the distribution of a \Levy process,
$p_{s,\pi}(y)$ is not, in general, the distribution of a \Levy process.
Nevertheless, one can consider the corresponding Bayesian predictive
\Levy density.
\qed 
\end{example}

\vspace{0.5cm}

The following lemma relates the Kullback--Leibler divergence between
mixtures of probability densities of pure jump subordinators to the
Kullback--Leibler divergence between the corresponding mixtures of \Levy densities.
No assumption is made that the parameter is introduced through an Esscher transform.

\vspace{0.5cm}

\begin{lemma}
\label{lemma2}
Let
\[
\mathcal P_s=\{p_{s,\xi}(y)\}
\]
be a parametric family of probability densities of the increment
$y_s=x_{t+s}-x_t$
of a pure jump subordinator.
Let $p_s(y)$ and $q_s(y)$ be mixture densities with respect to probability measures
$\pi_p$ and $\pi_q$ on the parameter space of $\xi$, respectively.

Assume that Conditions 1--3 below hold.
\begin{enumerate}
\item
The limits
\[
\tilde p(y)
=
\lim_{s\downarrow0}
\frac{1}{s}p_s(y),
\qquad
\tilde q(y)
=
\lim_{s\downarrow0}
\frac{1}{s}q_s(y)
\]
exist and define \Levy densities
$\tilde p(y)$ and $\tilde q(y)$.

\item Define
\[
f_{q,p}(s)
:=
\frac{q_s}{p_s}(0)
\qquad (s>0),
\]
and
\[
f_{q,p}(0)
:=
\lim_{s\downarrow0}
f_{q,p}(s).
\]
Then $f_{q,p}(0)=1$ and
$f_{q,p}(s)$ is one-sided differentiable at $s=0$.

\item
\begin{align}
\label{cond1}
\lim_{s \downarrow0}
\frac{1}{s}
\int
\{q_s(y)-f_{q,p}(s)p_s(y)\}
\,dy
&=
\int
\{\tilde q(y)-\tilde p(y)\}
\,dy,
\\
\label{cond2}
\lim_{s \downarrow0}
\frac{1}{s}
\int
p_s(y)
\log
\frac{q_s(y)}
     {f_{q,p}(s)p_s(y)}
\,dy
&=
\int
\tilde p(y)
\log
\frac{\tilde q(y)}
     {\tilde p(y)}
\,dy.
\end{align}
\end{enumerate}
Under these conditions,
\begin{align*}
\lim_{s \downarrow0}
\frac{1}{s}
\int
\left\{
q_s(y)-p_s(y)
-p_s(y)
\log
\frac{q_s(y)}
     {p_s(y)}
\right\}
dy
&=
\int
\left\{
\tilde q(y)-\tilde p(y)
-\tilde p(y)
\log
\frac{\tilde q(y)}
     {\tilde p(y)}
\right\}
dy
\\
&=
-f'_{q,p}(0)
-
\int
\tilde p(y)
\log
\frac{\tilde q(y)}
     {\tilde p(y)}
\,dy.
\end{align*}
\end{lemma}

\vspace{0.5cm}
\noindent
Proof.
\begin{align*}
\frac{1}{s} & \int 
\biggl\{q_s(y) - p_s(y) - p_s(y) \log \frac{q_s(y)}{p_s(y)} 
\biggr\} \dd y \\
&= \frac{1}{s} \int 
\biggl[ q_s(y) - f_{q,p}(s) p_s(y)
+ p_s(y) \Bigl\{f_{q,p}(s) - 1 -\log f_{q,p}(s) \bigr\}
- p_s(y) \log \frac{q_s(y)}{f_{q,p}(s) p_s(y)} 
\biggr] \dd y.
\end{align*}
Then,
\begin{align*}
\lim_{s \downarrow 0} \frac{1}{s} \int 
p_s(y) \Bigl\{f_{q,p}(s) - 1 -\log f_{q,p}(s) \bigr\}
\dd y
&=
\lim_{s \downarrow 0} \frac{1}{s}
\Bigl\{f_{q,p}(s) - 1 -\log f_{q,p}(s) \bigr\} \\
&= f'_{q,p}(0) - \frac{f'_{q,p}(0)}{f_{q,p}(0)} = 0.
\end{align*}
From \eqref{cond1} and \eqref{cond2},
\begin{align*}
\lim_{s \downarrow 0} & \frac{1}{s} \int 
\biggl\{q_s(y) - p_s(y) - p_s(y) \log \frac{q_s(y)}{p_s(y)} 
\biggr\} \dd y
= \int 
\biggl\{\tilde{q}(y) - \tilde{p}(y)
- \tilde{p}(y) \log \frac{\tilde{q}(y)}{\tilde{p}(y)} 
\biggr\} \dd y.
\end{align*}
Furthermore,
\begin{align*}
\lim_{s \downarrow 0} \frac{1}{s} \int 
\bigl\{q_s(y) - f_{q,p}(s) p_s(y) \bigr\} \dd y
=\lim_{s \downarrow 0} \frac{1}{s}
\bigl\{1 - f_{q,p}(s) \bigr\}
= -f'_{q,p}(0).
\end{align*}
\begin{flushright}
$\qed$
\end{flushright}

\vspace{0.5cm}

Note that
\begin{align}
\label{qcond1}
\lim_{s \downarrow 0}
\frac{1}{s}
\int
\bigl\{
q_s(y)-p_s(y)
\bigr\}
\,\dd y
&=
\int
\bigl\{
\tilde q(y)-\tilde p(y)
\bigr\}
\,\dd y,
\\
\label{qcond2}
\lim_{s \downarrow 0}
\frac{1}{s}
\int
p_s(y)
\log
\frac{q_s(y)}
     {p_s(y)}
\,\dd y
&=
\int
\tilde p(y)
\log
\frac{\tilde q(y)}
     {\tilde p(y)}
\,\dd y
\end{align}
do not hold in general.
This follows from the fact that the left-hand side of \eqref{qcond1} is identically zero.

\vspace{0.5cm}

The following theorem shows that infinitesimal prediction risk can be expressed as a Kullback--Leibler divergence
between \Levy measures.

\vspace{0.5cm}

\begin{theorem}
Assume the conditions of Lemma \ref{lemma2}.
The risk of infinitesimal prediction is given by
\begin{align}
\E_{x_t \mid \xi} & \left[\int \Bigl\{\tilde{p}_\pi (y \mid x_t) - \tilde{p}_\xi (y)
- \int \tilde{p}_\xi (y) \log
\frac
{\tilde{p}_\pi (y \mid x_t)}
{\tilde{p}_\xi (y)}
 \Bigr\} \dd y
\right]. 
\label{eq:16-1}
\end{align}
\end{theorem}
\vspace{0.5cm}

\begin{proof}
By Lemma~\ref{lemma2},
\[
\frac{1}{s}
\int
\left\{
p_{s,\pi}(y\mid x)
-
p_{s,\xi}(y)
-
p_{s,\xi}(y)
\log
\frac{p_{s,\pi}(y\mid x)}
     {p_{s,\xi}(y)}
\right\}
dy
\]
converges, as $s\downarrow0$, to
\[
\int
\left\{
\tilde p_\pi(y\mid x)
-
\tilde p_\xi(y)
-
\tilde p_\xi(y)
\log
\frac{\tilde p_\pi(y\mid x)}
     {\tilde p_\xi(y)}
\right\}
dy .
\]
Taking expectation with respect to $p_\xi(x)$
yields \eqref{eq:16-1}.
\end{proof}

\vspace{0.5cm}
\begin{example}
Consider the Gamma process $\mathrm{Ga}(\alpha t, \beta)$.
Let
\begin{align*}
p_s(y) \dd y &\coloneqq p_{s,\beta} (y) \frac{\dd y}{y}
= \frac{1}{\Gamma (\alpha s)} \frac{y^{\alpha s -1}}{\beta^{\alpha s}} \exp \left(- \frac{y}{\beta} \right) \dd y,
\end{align*}
and
\begin{align*}
q_s(y) \dd y &\coloneqq p_{s,\mathrm{J}} (y \mid x) \frac{\dd y}{y}
= \frac{\Gamma (\alpha + \alpha s)}{\Gamma (\alpha s) \Gamma (\alpha)}
\frac{x^{\alpha} y^{\alpha s}}{(x + y)^{\alpha + \alpha s}}  \frac{\dd y}{y}.
\end{align*}
Then,
\begin{align*}
\tilde{p}(y) \dd y &= \frac{\alpha}{y} \exp \left(- \frac{y}{\beta} \right) \dd y, ~~\mbox{and}~~~
\tilde{q}(y) \dd y = \frac{\alpha}{y} 
\biggl( \frac{x}{x + y} \biggr)^{\alpha} \dd y.
\end{align*}
Since
\[
\frac{q_s}{p_s}(y)
=
\frac{
\dfrac{\Gamma(\alpha+\alpha s)}
{\Gamma(\alpha s)\Gamma(\alpha)}
\dfrac{x^\alpha y^{\alpha s-1}}
{(x+y)^{\alpha+\alpha s}}
}{
\dfrac1{\Gamma(\alpha s)}
\dfrac{y^{\alpha s-1}}{\beta^{\alpha s}}
\exp\!\left(-\dfrac{y}{\beta}\right)
},
\]
we have
\[
f_{q,p}(s) = \frac{q_s}{p_s}(0)
=
\frac{\Gamma(\alpha+\alpha s)}
{\Gamma(\alpha)}
\left(\frac{\beta}{x}\right)^{\alpha s}.
\]

Then,
\begin{align*}
q_s(y) - f_{q,p}(s)p_s(y) &= \frac{\Gamma(\alpha+\alpha s)}
{\Gamma(\alpha s)\Gamma(\alpha)}
\frac{x^\alpha y^{\alpha s-1}}
{(x+y)^{\alpha+\alpha s}}
-
\frac{\Gamma(\alpha+\alpha s)}{\Gamma(\alpha)}
\left(\frac{\beta}{x}\right)^{\alpha s}
\frac1{\Gamma(\alpha s)}
\left(\frac{y}{\beta}\right)^{\alpha s-1}
\exp\!\left(-\frac{y}{\beta}\right) \\
&=
\frac{\Gamma(\alpha+\alpha s)}
{\Gamma(\alpha s)\Gamma(\alpha)}
\left(\frac{y}{x}\right)^{\alpha s-1}
\left\{
\frac1{(1+y/x)^{\alpha+\alpha s}}
-
\exp\!\left(-\frac{y}{\beta}\right)
\right\}
\end{align*}
Condition \eqref{cond1} is a consequence of the dominated convergence theorem.

Moreover,
\begin{align*}
p_s(y) \log \frac{q_s(y)}{f_{q,p}(s)p_s(y)}
&=
\frac{1}{s}
\frac{1}{\Gamma(\alpha s)}
\frac{y^{\alpha s-1}}{\beta^{\alpha s}}
\exp\!\left(-\frac{y}{\beta}\right)
\log
\frac{
\frac{\Gamma(\alpha+\alpha s)}
{\Gamma(\alpha s)\Gamma(\alpha)}
\frac{x^\alpha y^{\alpha s-1}}
{(x+y)^{\alpha+\alpha s}}
}{
\frac{\Gamma(\alpha+\alpha s)}{\Gamma(\alpha)}
\left(\dfrac{\beta}{x}\right)^{\alpha s}
\frac{1}{\Gamma(\alpha s)}
\frac{y^{\alpha s-1}}{\beta^{\alpha s}}
\exp\!\left(-\dfrac{y}{\beta}\right)
} \\
&=
\frac{\alpha}{\Gamma(\alpha s+1)}
\frac{(y/\beta)^{\alpha s}}{y}
\exp\!\left(-\frac{y}{\beta}\right)
\log
\left\{
\left(\frac{x}{x+y}\right)^{\alpha+\alpha s}
\exp\!\left(\frac{y}{\beta}\right)
\right\}
\end{align*}
Condition \eqref{cond2} follows from the dominated convergence theorem.

Therefore, for the Gamma process model,
\eqref{ipr} coincides \eqref{eq:16-1}.
\qed
\end{example}

\section{Mean mixture curvature}

As we have seen in the previous section, the behavior of infinitesimal prediction
in the Gamma model is markedly different from that in the normal and Poisson models.

The mean mixture curvature is an information-geometric quantity that characterizes
the asymptotic difference between a Bayesian predictive density and a plug-in predictive density.
The larger the mean mixture curvature, the larger the asymptotic improvement achieved
by Bayesian prediction over plug-in prediction \citep{K96}.

In this section,
we compare the mean mixture curvatures
of the normal, Poisson, and Gamma models.
For the definitions of the Fisher metric $g_{ij}$
and the mixture connection coefficients
$\Gamma^{(m)}_{ij,k}$,
see \cite{Amari85}.

\begin{example} Normal model

Consider the normal model
\begin{align*}
p_s(x \mid \mu)
=
\sqrt{\frac{s}{2 \pi}}
\exp \left\{
- \frac{s}{2} (x - \mu)^2
\right\}.
\end{align*}
The Fisher information is
$g_{\mu \mu} = s$.
Since $\mu$ is the expectation parameter,
$\mGamma{\mu \mu}{~~\mu} = 0$.
Moreover,
$\partial_\mu p_s = s (x - \mu) p_s$,
$\partial_\mu \partial_\mu p_s
=
\{s^2 (x - \mu)^2 - s \} p_s$,
and
$g^{\mu \mu} (\partial_\mu \partial_\mu p_s - \Gamma_{\mu \mu}^{~~\mu} \partial_\mu p_s)
=
\{s (x - \mu)^2 - 1 \} p_s$,
where $g^{\mu \mu} = 1/g_{\mu \mu}$.
Since the fourth moment is
$\E_{s,\mu}[(x-\mu)^4]=3/s^2$,
the mean mixture curvature is
\begin{align*}
\int \Bigl\{ g^{\mu\mu} (\partial_\mu \partial_\mu p_s - \Gamma_{\mu \mu}^{~~\mu} \partial_\mu p_s) \Bigr\}^2
\frac{1}{p_s} \dd x
=
\E_{s,\mu}
\Bigl[
\{s (x - \mu)^2 - 1 \}^2
\Bigr]
&=
s^2 \frac{3}{s^2} - 2 s \frac{1}{s} + 1 = 2.
\end{align*}
Thus,
the mean mixture curvature does not depend on $s$.
\qed
\end{example}

\vspace{0.5cm}
\begin{example}
Poisson model

Consider the Poisson model
\begin{align*}
p_s(x \mid \lambda)
=
\frac{(s \lambda)^x}{x !} \ee^{- s \lambda}.
\end{align*}
The Fisher information is
$g_{\lambda \lambda}
=s/\lambda$.
Since $\lambda$ is the expectation parameter, $\mGamma{\lambda \lambda}{~~ \lambda} = 0$.
Moreover,
\begin{align*}
\partial_\lambda p_s 
&=
\frac{s (s \lambda)^{x-1}}{(x - 1) !} \ee^{- s \lambda} \1_{x \geq 1}
 - s \frac{(s \lambda)^x}{x !} \ee^{- s \lambda}, \\
\partial_\lambda \partial_\lambda p_s 
&=
\frac{s^2 (s \lambda)^{x - 2}}{(x - 2) !} \ee^{- s \lambda} \1_{x \geq 2}
- 2 \frac{s^2 (s \lambda)^{x-1}}{(x - 1) !} \ee^{- s \lambda} \1_{x \geq 1}
+ s^2 \frac{(s \lambda)^x}{x !} \ee^{- s \lambda}.
\end{align*}
and
\[
g^{\lambda \lambda}
\Bigl( \partial_\lambda \partial_\lambda p_s -
\mGamma{\lambda \lambda}{~~\lambda}
\partial_\lambda p_s \Bigr)
=
\frac{ (s \lambda)^{x - 1}}{(x - 2) !} \ee^{- s \lambda}
\1_{x \geq 2}
- 2 \frac{(s \lambda)^{x}}{(x - 1) !}
\ee^{- s \lambda} \1_{x \geq 1}
+ \frac{(s \lambda)^{x+1}}{x!} \ee^{- s \lambda},
\]
where $g^{\lambda \lambda} = 1/g_{\lambda \lambda}$.
Since
$\E_{s,\lambda} [ x^2 ] = s\lambda (1 + s\lambda)$,
$\E_{s,\lambda} [ x^3 ] = s\lambda (1 + 3 s\lambda + s^2\lambda^2)$,
and $\E_{s,\lambda} [x^4] = s\lambda (1 + 7 s\lambda + 6 s^2\lambda^2 + s^3\lambda^3)$,
the mean mixture curvature is
\begin{align*}
\sum_x &
\Bigl\{
g^{\lambda \lambda}
\bigl(
\partial_\lambda \partial_\lambda p_s
-
\mGamma{\lambda \lambda}{~~\lambda}
\partial_\lambda p_s
\bigr)
\Bigr\}^2 \frac{1}{p_s}
=\sum_x
\left\{
(s \lambda)^{-1} x(x - 1)
- 2 x
+ s \lambda
\right\}^2
\frac{(s \lambda)^x}{x !} \ee^{- s \lambda} \\
&= \E_{s,\lambda}
\left[
\frac{x^2 (x - 1)^2}{(s \lambda)^2}
+ 4x^2
+ (s\lambda)^2
- \frac{4}{s \lambda} x^2 (x - 1)
+ 2 x (x -1)
- 4(s \lambda)x
\right]
= 2.
\end{align*}
Hence, the mean mixture curvature does not depend on $s$.
\qed
\end{example}

\vspace{0.5cm}
\begin{example}
Gamma model

Consider the Gamma model
\begin{align*}
p_s(x \mid \beta) \frac{\dd x}{x}
&= \frac{1}{\Gamma (s)}
\frac{x^s}{\beta^s}
\exp \left(- \frac{x}{\beta} \right)
\frac{\dd x}{x}.
\end{align*}
The Fisher information is $g_{\beta \beta} = s/\beta^2$.
Since $\beta$ is the expectation parameter,
$\mGamma{\beta \beta}{~ \beta} = 0$.

Moreover,
\begin{align*}
\partial_\beta p_s
&=
\left(
\frac{x}{\beta^2} - \frac{s}{\beta}
\right) p_s, \\
\partial_\beta \partial_\beta p_s
&=
\left\{
- 2 \frac{x}{\beta^3}
+ \frac{s}{\beta^2}
+ \left(
\frac{x}{\beta^2} - \frac{s}{\beta}
\right)^2
\right\} p_s
=
\left\{
\frac{x^2}{\beta^4}
- \frac{2}{\beta^3} (s + 1) x
+ \frac{1}{\beta^2} (s^2 + s)
\right\} p_s,
\end{align*}
and
\begin{align*}
g^{\beta \beta}
\Bigl(\partial^2_\beta p_s
-
\mGamma{\beta \beta}{~\beta} \partial_\beta p_s \Bigr)
&=
g^{\beta \beta} \partial_\beta \partial_\beta p_s
=
\left\{
\frac{x^2}{s \beta^2}
- \frac{2}{\beta} \left(1 + \frac{1}{s} \right) x
+ (s + 1)
\right\} p_s,
\end{align*}
where $g^{\beta \beta} = 1/g_{\beta \beta}$.

The moments are given by
\begin{align*}
\E_{s,\beta} [x^k]
&=
\int x^k
\frac{1}{\Gamma (s)}
\frac{x^s}{\beta^s}
\exp \left(- \frac{x}{\beta} \right)
\frac{1}{x} \dd x
=
\beta^k \frac{\Gamma (k + s)}{\Gamma (s)}
=
(k + s - 1) \dotsb s \beta^k.
\end{align*}
In particular,
$\E_{s,\beta} [x] = s \beta$,
$\E_{s,\beta} [x^2] = \beta^2 (s^2 + s)$,
$\E_{s,\beta} [x^3] = (s^3 + 3 s^2 + 2 s) \beta^3$, and
$\E_{s,\beta} [x^4] = (s^4 + 6 s^3 + 11 s^2 + 6s) \beta^4$.

Hence,
the mean mixture curvature is
\begin{align*}
\E_{s,\beta}
\biggl[
\Bigl\{
\frac{x^2}{s \beta^2}
- \frac{2}{\beta}
\Bigl(1 + \frac{1}{s} \Bigr) x
+ (s + 1)
\Bigr\}^2
\biggr]
=
2 \frac{1+s}{s}.
\end{align*}
Therefore,
the mean mixture curvature diverges to infinity
as $s \downarrow 0$.
\qed
\end{example}

\vspace{0.5cm}

The substantially different behavior of the mean mixture curvature
in the Gamma model,
compared with the normal and Poisson models,
particularly when the amount of observational information is small,
may reflect the qualitative difference
between prediction and estimation in these models.

\section{Conclusion}

In this paper, we investigated estimation and prediction in Gamma models under the Kullback--Leibler loss.

Our main result is that the Bayesian predictive density based on the Jeffreys prior is admissible for all $\alpha>0$. This resolves the admissibility problem for Bayesian predictive densities in Gamma models. As a related result, we also established the admissibility of the corresponding Bayesian estimator for $\alpha>1$.

To prove the predictive admissibility result, we developed an infinitesimal prediction framework based on Gamma processes. This framework led naturally to a Kullback--Leibler loss for \Levy densities and to the notion of a Bayesian predictive \Levy density. We showed that the Bayesian predictive \Levy density is given by the posterior mean \Levy density under the corresponding loss.

Finally, we discussed this phenomenon from an information-geometric viewpoint through the mean mixture curvature. The results suggest that infinitesimal prediction provides a useful perspective for understanding the interplay among Bayesian prediction, infinitely divisible distributions, and information geometry.

\section*{Acknowledgements}

This work was supported by JSPS KAKENHI Grant Numbers 22H00510 and 25H01156.

\end{document}